\documentclass[11pt]{article}
\usepackage{amsmath,amssymb,amsthm,amscd}

\newtheorem{theorem}{Theorem}[section]
\newtheorem{lemma}[theorem]{Lemma}
\newtheorem{proposition}[theorem]{Proposition}
\newtheorem{corollary}[theorem]{Corollary}

\theoremstyle{plain}

\theoremstyle{definition}
\newtheorem{definition}[theorem]{Definition}
\newtheorem{remark}[theorem]{Remark}

\numberwithin{equation}{section}

\renewcommand{\labelenumi}{\textup{(\theenumi)}}

\renewcommand{\phi}{\varphi}

\newcommand{\Homeo}{\operatorname{Homeo}}

\newcommand{\id}{\operatorname{id}}
\newcommand{\Ker}{\operatorname{Ker}}

\newcommand{\Ad}{\operatorname{Ad}}

\newcommand{\K}{\mathcal{K}}
\newcommand{\C}{\mathcal{C}}

\newcommand{\N}{\mathbb{N}}
\newcommand{\R}{\mathbb{R}}
\newcommand{\T}{\mathbb{T}}
\newcommand{\Z}{\mathbb{Z}}

\title{State splitting,  strong shift equivalence and stable isomorphism of Cuntz--Krieger algebras}
\author{Kengo Matsumoto \\
Department of Mathematics \\
Joetsu University of Education \\
Joetsu, 943-8512, Japan
}

\date{ }

\begin{document}
\maketitle

\def\det{{{\operatorname{det}}}}

\begin{abstract}
We  prove that if two nonnegative matrices are strong shift equivalent, 
the associated stable Cuntz--Krieger algebras with generalized gauge actions
are  conjugate.
The proof is done by a purely functional analytic method and based on constructing imprimitivity bimodule from bipartite directed graphs through strong shift equivalent matrices,
so that we may clarify K-theoretic behavior of the stable conjugacy 
between the associated stable Cuntz--Krieger algebras.
We also examine our machinery for the matrices obtained by state splitting graphs, 
so that topological conjugacy of the topological Markov shifts is 
described in terms of some equivalence relation of the Cuntz--Krieger algebras 
with canonical masas and the gauge actions without stabilization. 
\end{abstract}

{\it Mathematics Subject Classification}:
 Primary 46L55; Secondary 37B10, 46L35.

{\it Keywords and phrases}:
Topological Markov shifts, strong shift equivalence,  Cuntz--Krieger algebras,
state splitting


\def\OA{{{\mathcal{O}}_A}}
\def\OB{{{\mathcal{O}}_B}}
\def\OZ{{{\mathcal{O}}_Z}}
\def\V{{\mathcal{V}}}
\def\E{{\mathcal{E}}}
\def\OTA{{{\mathcal{O}}_{\tilde{A}}}}
\def\SOA{{{\mathcal{O}}_A}\otimes{\mathcal{K}}}
\def\SOB{{{\mathcal{O}}_B}\otimes{\mathcal{K}}}
\def\SOZ{{{\mathcal{O}}_Z}\otimes{\mathcal{K}}}
\def\SOTA{{{\mathcal{O}}_{\tilde{A}}\otimes{\mathcal{K}}}}
\def\DA{{{\mathcal{D}}_A}}
\def\DB{{{\mathcal{D}}_B}}
\def\DZ{{{\mathcal{D}}_Z}}
\def\DTA{{{\mathcal{D}}_{\tilde{A}}}}
\def\SDA{{{\mathcal{D}}_A}\otimes{\mathcal{C}}}
\def\SDB{{{\mathcal{D}}_B}\otimes{\mathcal{C}}}
\def\SDZ{{{\mathcal{D}}_Z}\otimes{\mathcal{C}}}
\def\SDTA{{{\mathcal{D}}_{\tilde{A}}\otimes{\mathcal{C}}}}
\def\BC{{{\mathcal{B}}_C}}
\def\BD{{{\mathcal{B}}_D}}
\def\OAG{{\mathcal{O}}_{A^G}}
\def\OBG{{\mathcal{O}}_{B^G}}
\def\O2{{{\mathcal{O}}_2}}
\def\D2{{{\mathcal{D}}_2}}

\newcommand{\mathP}{\mathcal{P}}
\def\OAG{{\mathcal{O}}_{A^G}}
\def\DAG{{\mathcal{D}}_{A^G}}
\def\OBG{{\mathcal{O}}_{B^G}}
\def\OAGP{{\mathcal{O}}_{A^{{[\mathP]}}}} 
\def\DAGP{{\mathcal{D}}_{A^{{[\mathP]}}}} 
\def\OAGHP{{\mathcal{O}}_{{Z^{[\mathP]}}}} 
\def\DAGHP{{\mathcal{D}}_{{Z^{[\mathP]}}}} 
\def\AGP{A^{{[\mathP]}}} 
\def\AGHP{{Z^{[\mathP]}}} 
\def\OAGIP{{\mathcal{O}}_{A_{[\mathP]}}} 
\def\DAGIP{{\mathcal{D}}_{A_{[\mathP]}}} 
\def\OAGHIP{{\mathcal{O}}_{{Z_{[\mathP]}}}} 
\def\DAGHIP{{\mathcal{D}}_{{Z_{[\mathP]}}}} 
\def\AGIP{A_{[\mathP]}} 
\def\AGHIP{{Z_{[\mathP]}}}
\def\CKT{{\mathcal{T}}}
\def\Max{{{\operatorname{Max}}}}
\def\Per{{{\operatorname{Per}}}}
\def\PerB{{{\operatorname{PerB}}}}
\def\Homeo{{{\operatorname{Homeo}}}}
\def\HA{{{\frak H}_A}}
\def\HB{{{\frak H}_B}}
\def\HSA{{H_{\sigma_A}(X_A)}}
\def\Out{{{\operatorname{Out}}}}
\def\Aut{{{\operatorname{Aut}}}}
\def\Ad{{{\operatorname{Ad}}}}
\def\Inn{{{\operatorname{Inn}}}}
\def\det{{{\operatorname{det}}}}
\def\exp{{{\operatorname{exp}}}}
\def\cobdy{{{\operatorname{cobdy}}}}
\def\Ker{{{\operatorname{Ker}}}}
\def\ind{{{\operatorname{ind}}}}
\def\id{{{\operatorname{id}}}}
\def\supp{{{\operatorname{supp}}}}
\def\co{{{\operatorname{co}}}}
\def\Sco{{{\operatorname{Sco}}}}
\def\ActA{{{\operatorname{Act}_{\DA}(\mathbb{T},\OA)}}}
\def\ActB{{{\operatorname{Act}_{\DB}(\mathbb{T},\OB)}}}
\def\RepOA{{{\operatorname{Rep}(\mathbb{T},\OA)}}}
\def\RepDA{{{\operatorname{Rep}(\mathbb{T},\DA)}}}
\def\RepDB{{{\operatorname{Rep}(\mathbb{T},\DB)}}}
\def\U{{{\mathcal{U}}}}

\section{Introduction}
Suppose $1<N\in \N$.
Let
$A=[A(i,j)]_{i,j=1}^N$
be an $N \times N$ 
irreducible and not any permutation matrix
with entries in $\{0,1\}$.
Throughout the paper, 
we assume that matrix $A$ is  
irreducible and not any permutation.
The Cuntz--Krieger algebra
$\OA$ associated to the matrix $A$
 is a universal purely infinite simple $C^*$-algebra generated by
partial isometries
$S_1,\dots,S_N$
satisfying the relations:
\begin{equation} 
\sum_{j=1}^N S_j S_j^* = 1, \qquad
S_i^* S_i = \sum_{j=1}^N A(i,j) S_jS_j^*
\quad 
\text{ for } i=1,\dots,N. \label{eq:CK}
\end{equation} 
An action of the circle group 
${\mathbb{R}}/\Z = {\mathbb{T}}$
on the algebra $\OA$ is defined by 
the map
$S_i \rightarrow e^{2 \pi\sqrt{-1}t}S_i,
\, i=1,\dots,N$ 
for
 $t \in {\mathbb{R}}/\Z = {\mathbb{T}}$.
It is denoted  by $\rho^A_t$
and  called the gauge action.
Cuntz and Krieger in \cite{CK} have shown that the algebra 
$\OA$ has close relationships with the underlying topological
dynamical system called topological Markov shift.
Let us denote by 
$\bar{X}_A$ the shift space 
\begin{equation}
\bar{X}_A = \{ (x_n )_{n \in \Z} \in \{1,\dots,N \}^{\Z}
\mid
A(x_n,x_{n+1}) =1 \text{ for all } n \in {\Z}
\} \label{eq:twoMarkovshift}
\end{equation}
which is endowed with a relative topology of the product topology
in $\{1,\dots,N \}^{\Z}$, 
so that $\bar{X}_A$ 
is a compact Hausdorff space.
Define the shift transformation $\bar{\sigma}_A$ 
on $\bar{X}_A$
by $\bar{\sigma}_{A}((x_n)_{n \in {\Z}})=(x_{n+1} )_{n \in \Z}$,
which is a homeomorphism on $\bar{X}_A$.
The topological dynamical system $(\bar{X}_A,\bar{\sigma}_A)$
is called the two-sided topological Markov shift for matrix $A$.
The (right) one-sided topological Markov shift 
$(X_A, \sigma_A)$
is similarly defined by the shift space
\begin{equation}
X_A = \{ (x_n )_{n \in \N} \in \{1,\dots,N \}^{\N}
\mid
A(x_n,x_{n+1}) =1 \text{ for all } n \in {\N}
\} \label{eq:oneMarkovshift}
\end{equation}
with the shift transformation
$\sigma_{A}((x_n)_{n \in {\N}})=(x_{n+1} )_{n \in \N}$
on $X_A$, which is a continuous surjection.
 Let us consider the $C^*$-subalgebra 
$\DA$ of $\OA$
 generated by the projections of the form:
$S_{\mu_1}\cdots S_{\mu_n}S_{\mu_n}^* \cdots S_{\mu_1}^*,
\mu_1,\dots, \mu_n =1,\dots,N$.
Let  
$\ell^2(\N)$ 
be the  separable infinite dimensional  Hilbert space.
Let us denote by $\K$ the $C^*$-algebra $\K(\ell^2(\N))$ 
of compact operators on $\ell^2(\N)$ and 
by ${\mathcal{C}}$ its maximal abelian $C^*$-subalgebra consisting of diagonal operators
on $\ell^2(\N)$. 
Cuntz and Krieger have shown in \cite{CK} that the following theorem:
\begin{theorem}[{\cite[3.8 Theorem]{CK}}, cf. \cite{CR}, \cite{Cu3}] \label{thm:CK}
Let $A, B$ be irreducible and non permutation matrices with entries in
$\{0,1\}$.
 Suppose that two-sided topological Markov shifts 
$(\bar{X}_A, \bar{\sigma}_A)$ and 
$(\bar{X}_B, \bar{\sigma}_B)$ 
are topologically conjugate.
Then there exists an isomorphism
$\Phi:\SOA \rightarrow \SOB$ such that 
\begin{equation*}
\Phi(\SDA) = \SDB, \qquad 
\Phi \circ (\rho^{A}_t\otimes \id)
 = (\rho^{B}_t\otimes\id) \circ \Phi \quad
\text{ for }
 t \in \T.
\end{equation*}
\end{theorem}
The proof given in the paper \cite{CK} was based on dynamical method.
Recently, T. M. Carlsen and J. Rout in \cite{CR} have generalized the above result to 
a wider class of graph algebras with more general gauge actions
by using groupoid technique  (cf. \cite{AER}, \cite{CRS}). 
They have also shown that the converse implication
for the assertion of Theorem \ref{thm:CK} also holds.

In the first half of the paper, 
we will give a different proof of Theorem \ref{thm:CK}
from the above Cuntz--Krieger and Carlsen-Rout methods.
Their methods to prove Theorem \ref{thm:CK} were basically
due to topological and groupoid techniques. 
Our method given in this paper is a $C^*$-algebraic one.
The idea of our proof is essentially due to the arguments of 
the author's previous  papers \cite{MaETDS2004},  \cite{MaPre2015},
\cite{MaPre2016}
in which 
imprimitivity bimodules obtained from strong shift equivalence matrices
have been studied. 
Related approaches to ours are seen in 
\cite{EKaKa}, \cite{Bates}, \cite{BP}, \cite{KaKaJFA2014}, \cite{MS},
\cite{MPT}, \cite{Tomforde}, etc.
The key of our proof is to use the following Williams' strong shift equivalence theory
which classifies two-sided topological Markov shifts in terms of the defining matrices.
 (\cite{Williams}).
Two nonnegative matrices $A, B$ are said to be elementary equivalent
if $A = CD, B = DC$ for some
 nonnegative rectangular matrices $C,D$.
R. F. Williams in \cite{Williams}
 introduced the notion of  strong shift equivalence of matrices
in the following way.
Two matrices  $A$ and $B$ are said to be strong shift equivalent
if there exists a finite sequence of nonnegative matrices $A_0,A_1,\dots, A_n$
such that
$A = A_0, B = A_n$ and $A_i$ is elementary equivalent to $A_{i+1}$
such as
$A_i = C_i D_i, \, A_{i+1} = D_i C_i
$
 for 
$i=0,1,\dots, n-1$.  
The situation is written 
\begin{equation}
A \underset{C_1,D_1}{\approx}\cdots \underset{C_{n},D_{n}}{\approx}B. \label{eq:SSE}
\end{equation}
Williams  proved in \cite{Williams} that 
the two-sided topological Markov shifts 
$(\bar{X}_A, \bar{\sigma}_A)$ and 
$(\bar{X}_B, \bar{\sigma}_B)$ 
are topologically conjugate
if and only if the matrices $A$ and $B$ are strong shift equivalent.
 Hence elementary equivalence generates topological conjugacy of two-sided topological Markov shifts. 

The $C^*$-subalgebra $\DA$ of the Cuntz--Krieger algebra $\OA$ is identified with the abelian $C^*$-algebra $C(X_A)$ of the complex valued continuous functions on $X_A$
by regarding the projection
$S_{\mu_1}\cdots S_{\mu_n}S_{\mu_n}^* \cdots S_{\mu_1}^*$
with the characteristic function
$\chi_{U_{\mu_1\cdots \mu_n}} \in C(X_A)$
of the cylinder set
$U_{\mu_1\cdots \mu_n}$
for a word ${\mu_1\cdots \mu_n}$.
We denote by 
$C(X_A, \Z)$ the set of $\Z$-valued continuous functions on $X_A$.
The gauge action on $\OA$ is generalized in the following way (\cite{MaPre2015}).
For $f \in C(X_A, \Z)$,
define an automorphism
$\rho_t^{A,f}$ on $\OA$ for each $t \in \T$ 
by 
\begin{equation}
\rho_t^{A,f}(S_i) = U_t(f) S_i
\qquad i=1,\dots,N \label{eq:rhotf}
\end{equation}
where
$U_t(f), t \in \T= \R/\Z$ 
is defined  by the unitary 
$U_t(f) = \exp({2\pi \sqrt{-1} t f})
$
in $\DA$.
Now 
suppose that
$A$ and $B$ are elementary equivalent such that 
$A = CD$ and $B =DC$ for some nonnegative rectangular matrices $C,D$.
As in \cite[Section 4]{MaPre2015},
we may find canonical
homomorphisms
$\phi:C(X_A,\Z) \rightarrow  C(X_B,\Z)$
and
$\psi:C(X_B,\Z) \rightarrow  C(X_A,\Z)$
of groups such that 
\begin{equation}
(\psi \circ \phi)(f) = f \circ \sigma_A,\qquad
(\phi \circ \psi)(g) = g \circ \sigma_B \label{eq:psiphi}
\end{equation}
for $f \in C(X_A,\Z)$ and $g \in C(X_B,\Z)$.
Cuntz has proved that
there exists an  
isomorphism
$\epsilon_A: K_0(\OA) \rightarrow \Z^N/{(\id - A^{t})\Z^N}$
such that  
$\epsilon_A([1_A]) = [(1,1,\dots,1)]$,
where $1_A$ is the unit of $\OA$
(\cite[3.1 Proposition]{Cu3}).
We will first prove the following theorem.
\begin{theorem}[Theorem \ref{thm:main1} and Theorem \ref{thm:KC}] \label{thm:1.2}
Let $A, B$ be irreducible and non permutation matrices.
Suppose that $A$ and $B$ are elementary equivalent such that
$A = CD$ and $B =DC$
for some nonnegative rectangular matrices $C,D$.
Then there exists an isomorphism
$\Phi:\SOA \rightarrow \SOB$ satisfying 
$\Phi(\SDA) = \SDB$ such that
\begin{equation*}
\Phi  \circ (\rho^{A,\psi(g)}_t\otimes \id)
 = (\rho^{B,g}_t\otimes\id) \circ \Phi
\quad
\text{ for }
g \in C(X_B,\Z), \, 
 t \in \T
\end{equation*}
and the diagram 
$$
\begin{CD}
K_0(\OA) @>\Phi_* >> K_0(\OB) \\
@V{\epsilon_A }VV  @VV{\epsilon_B}V \\
\Z^N/{(\id - A^{t})\Z^N} @> \Phi_{C^t} >> \Z^M/{(\id - B^{t})\Z^M} 
\end{CD}
$$
is commutative,
where $A$ is an $N\times N$ matrix,
$B$ is an $M\times M$ matrix  and
 $\Phi_{C^t}$ is an isomorphism of groups induced by multiplying the matrix
$C^t: \Z^N\longrightarrow \Z^M$. 
\end{theorem}
As a direct corollary of the above theorem, 
we obtain  Theorem \ref{thm:CK}. 
Since any $C^*$-algebraic proof of Theorem \ref{thm:CK} has not been known until now,
the proof of Theorem \ref{thm:main1} given in this paper confirms Theorem \ref{thm:CK} in terms of $C^*$-algebras
(cf. \cite{MaETDS2004}, \cite{MaPre2015}, \cite{MaPre2016}). 
Moreover, as in the following corollary, the K-theoretic behavior of the isomorphism 
$\Phi$ can be described in terms of the matrices
$C_1,\dots,C_n$ connecting $A$ and $B$. 
\begin{corollary}
\label{cor:main}
Let $A, B$ be irreducible and non permutation matrices.
Suppose that $A$ and $B$ are strong shift equivalent such that
$A \underset{C_1,D_1}{\approx}\cdots \underset{C_{n},D_{n}}{\approx}B
$
for some nonnegative rectangular matrices
$C_i, D_i, i=1,2,\dots,n.$
Then there exists an isomorphism
$\Phi:\SOA \rightarrow \SOB$ satisfying 
$\Phi(\SDA) = \SDB$ such that
\begin{equation*}
\Phi  \circ (\rho^{A}_t\otimes \id)
 = (\rho^{B}_t\otimes\id) \circ \Phi
\quad
\text{ for }
 t \in \T
\end{equation*}
and the diagram 
$$
\begin{CD}
K_0(\OA) @>\Phi_* >> K_0(\OB) \\
@V{\epsilon_A }VV  @VV{\epsilon_B}V \\
\Z^N/{(\id - A^{t})\Z^N} @> \Phi_{{(C_1\cdots C_n)}^t} >> \Z^M/{(\id - B^{t})\Z^M} 
\end{CD}
$$
is commutative,
where $A$ is an $N\times N$ matrix,
$B$ is an $M\times M$   and
$\Phi_{{(C_1\cdots C_n)}^t}$ is an isomorphism induced by multiplying the matrix
${(C_1\cdots C_n)}^t: \Z^N\longrightarrow \Z^M$. 
\end{corollary}

In his proof of the Williams's theorem, he has introduced 
a key procedure for constructing new directed graphs from an original directed graph.
He has then shown that the underlying matrices are connected  by a finite sequence
of the directed graphs obtained by the procedures for topologically conjugate topological Markov shifts (\cite{Williams}, cf. \cite{Kitchens}, \cite{LM}).
The key procedure is called state splitting. 
The converse procedure is called state amalgamation.
The transition matrices obtained by state splitting
are strong shift equivalent.
In the second half of the paper,  
we  will examine our construction of the isomorphism
$\Phi: \SOA\longrightarrow \SOB$ in Theorem \ref{thm:1.2} and 
Corollary \ref{cor:main} for state splitting matrices. 
Let $G=(\V,\E)$ be the directed graph for the given nonnengative matrix $A$.
$\AGP$is the transition matrix of the out-split graph of  $G$ 
by a partition $\mathP$ of the edges $\E$.
Then there exists an isomorphism
$\Phi^{[\mathP]}:\OA \rightarrow \OAGP$ satisfying 
$\Phi^{[\mathP]}(\DA) = \DAGP$ such that
\begin{equation}
\Phi^{[\mathP]}  \circ \rho^{A}_t
 = \rho^{\AGP}_t \circ \Phi^{[\mathP]}
\quad
\text{ for } 
 t \in \T. \label{eq:PhirhoOut1}
\end{equation}
(Theorem \ref{thm:mainout}).
In the cases of in-splitting we need to stabilize the algebras $\OA$ and $\OB$
to obtain a similar result to the above equality \eqref{eq:PhirhoOut1}
(Theorem \ref{thm:mainin}).
We call the triplet
$(\OA,\DA,\rho^A)$
the Cuntz--Krieger triplet and write it 
$\CKT_A$.
It has been proved that 
$\CKT_A$ and $\CKT_B$ are isomorphic
 if and only if 
their underlying one-sided topological Markov shifts
$(X_A,\sigma_A)$ and
$(X_B,\sigma_B)$
are eventually conjugate
(\cite{MaPre2015}, \cite{MaPAMS2016}). 
We denote by
$\bar{\CKT}_A$
the pair 
$(\CKT_{A^t}, \CKT_A)$
of the Cuntz--Krieger triplets.
Since  
an in-split graph is obtained by out-splitting of  the transposed graph,   
it seems to be reasonable to say that  
$\bar{\CKT}_A$
and
$\bar{\CKT}_B$
are  {\it transpose free isomorphic in }$1$-{\it step}
if $\CKT_A$ and $\CKT_B$ are isomorphic
or
$\CKT_{A^t}$ and $\CKT_{B^t}$ are isomorphic.
More generally 
we say that  
$\bar{\CKT}_A$
and
$\bar{\CKT}_B$
are  {\it transpose free isomorphic}
if they are connected by an equivalence relation generated by 
 transpose free isomorphisms in $1$-step.
We then prove  the following theorem.
\begin{theorem}[{Theorem \ref{thm:transfree}}]
Suppose that $A$ and $B$ are nonnegative irreducible matrices.
Then two-sided topological Markov shifts 
$(\bar{X}_A,\bar{\sigma}_A)$ and $(\bar{X}_B,\bar{\sigma}_B)$
are topologically conjugate 
if and only if
$\bar{\CKT}_A$
and
$\bar{\CKT}_B$
are transpose free isomorphic.
\end{theorem}
This gives a  characterization 
of topologically 
conjugate two-sided topological Markov shift in terms of the Cuntz--Krieger
triplets without stabilization. 
\section{Strong shift equivalence }
In what follows, we suppose that  $A =[A(i,j)]_{i,j=1}^N $ 
is an $N\times N$ matrix with entries in nonnegative integers.
Let us consider $N$ vertices
$\{ I^A_1, \dots, I_N^A\}$
which is denoted by $\V_A$.
We consider 
$A(i,j)$ directed edges  from $I_i^A$ to $I_j^A$.
The set of edges is denoted by
$\E_A$.
We then have a 
directed graph 
$G_A = (\V_A,\E_A)$ for the matrix $A$.
For an directed edge
$a_j \in \E_A$,
its target vertex and source vertex are denoted by
   $t(a_i),  s(a_i)$, respectively.
We then have the associated matrix 
$A^G$ with entries in $\{0,1\}$
by setting
\begin{equation}
A^G(i, j) =
\begin{cases} 
 1 &  \text{  if  } t(a_i) = s(a_j), \\
 0 & \text{  otherwise}
\end{cases} \label{eq:AG}
\end{equation}
for $i,j =1,\dots,N_A$,
which expresses the transition of  the directed edges
of $\E_A$. 
The topological Markov shifts $(\bar{X}_A,\bar{\sigma}_A)$ and
$(X_A,\sigma_A)$ for the nonnegative matrix $A$ 
are defined as those of 
 $(\bar{X}_{A^G},\bar{\sigma}_{A^G})$ and
$(X_{A^G},\sigma_{A^G})$, respectively. 
The Cuntz--Krieger algebra $\OA$ for the matrix $A$  
 is defined as the Cuntz--Krieger algebra
 ${\mathcal{O}}_{A^G}$ for the matrix $A^G$ 
which is the universal $C^*$-algebra generated by 
partial isometries
$S_{a_i}$ indexed by edges $a_i, i=1,\dots, N_A$ subject to the relations
\eqref{eq:CK} for the matrix $A^G$ instead of $A$.

Let us assume that nonnegative irreducible matrices $A, B$ 
 are elementary equivalent so that
$A = CD$ and $B=DC$
for some  nonnegative rectangular matrices $C, D$
such that 
$C$ is an $N \times M$ matrix and 
$D$ is an $M \times N$ matrix, respectively. 
 We set the $(N+M)\times(N+ M)$ matrix
$
Z =
\begin{bmatrix}
0 & C \\
D & 0
\end{bmatrix}.
$
The directed graph
$G_Z =(\V_Z,\E_Z)$
for the matrix $Z$ 
is a bipartite graph such that 
$\E_Z = \E_C\cup \E_D$
where 
$\E_C, \E_D$ are the edges corresponding to the matrix entries of
$C, D$ respectively.
As
$A= CD$ (resp. $B=DC$),
the edge set $\E_A$ (resp. $\E_B$)
is identified with a subset 
of the pairs  of edges
 $\E_C$ (resp. $\E_D$) and $\E_D$ (resp. $\E_C$). 
Hence 
we identify
 an edge $a$ of $\E_A$ with a pair $c(a) d(a)$ of
edges $c(a) \in \E_C$ and $d(a)\in \E_D$.
Similarly 
we identify
 an edge $b$ of $\E_B$ with a pair $d(b) c(b)$ of
edges $d(b) \in \E_D$ and $c(b)\in \E_C$.

We will consider  the Cuntz--Krieger algebra 
$\OZ$
for the nonnegative matrix $Z$.
The canonical generating partial isometries 
are denoted by
$ S_c, S_d, c \in \E_C, d \in \E_D$ 
which are indexed by edges of $\E_C$ and of $\E_D$
satisfying the following relations
\begin{gather*}
\sum_{c \in \E_C} S_c S_c^*
+
\sum_{d \in \E_D} S_d S_d^*
=1, \\
S_c^* S_c = \sum_{d \in \E_D}Z(c,d) S_d S_d^*, \qquad
S_d^* S_d = \sum_{c \in \E_C}Z(d,c) S_c S_c^*
\end{gather*}
for $c \in \E_C, d \in \E_D$.
Put  the projections
$P_A$ and $P_B$ 
 in $\OZ$ by 
$P_A = \sum_{c \in \E_C} S_c S_c^*$
and
$P_B = \sum_{d \in \E_D} S_d S_d^*$
so that $P_A + P_B =1$.
Under the identifications
between 
$\E_A$ (resp. $\E_B$)
and
$\{ c(a) d(a) \in \E_C \E_D \mid a \in \E_A\}$
(resp. $\{ d(b) c(b) \in \E_D\E_C \mid b \in \E_B\}$),
we write
$S_{cd} = S_a$ (resp. $S_{dc} = S_b$)
where $S_{cd}$ denotes $S_c S_d$ (resp. $S_{dc}$ denotes $S_d S_c$)
if 
$c=c(a), d=d(a)$ 
(resp.
$d=d(b), c=c(b)$).
The $C^*$-subalgebra
\begin{align*}
& C^*(S_a : a = cd \text{ for some } c \in C, d \in D ) \\
(\text{resp. } & C^*(S_b : b = dc \text{ for some } d \in D, c \in C ) )
\end{align*}
of $\OZ$ coincides with 
$\OA$ (resp. $\OB$).
By \cite{MaETDS2004} (cf. \cite{MaPre2015}),
we know that
\begin{equation}
P_A \OZ P_A = \OA, \qquad P_B \OZ P_B = \OB,
\qquad
 \DZ P_A = \DA, \qquad  \DZ P_B = \DB. \label{eq:4.1}
\end{equation}
As in \cite[Lemma 3.1]{MaETDS2004},
  $P_A \OZ P_B$ has a structure of 
 $\OA-\OB$-imprimitivity bimodule in a natural way (cf. \cite{Brown}, \cite{BGR}, \cite{Rieffel1}).

Let us denote by $\{ \delta_j\}_{j \in \N}$ the complete orthonormal basis
of the Hilbert space 
$\ell^2(\N)$
defined by
\begin{equation}
\delta_j(m)
=
\begin{cases}
1 & \text{ if } j=m, \\
0 & \text{ otherwise, }
\end{cases} \label{eq:delta}
\end{equation}
for $m \in \N$.
Let us denote by $e_{i,j}, i,j\in \N$ the matrix units on $\ell^2(\N)$
such that $e_{i,j}\delta_j = \delta_i$.
Recall that the $C^*$-algebra generated by all of them is denoted by $\K$
which is the $C^*$-algebra of compact operators on $\ell^2(\N)$.
Its multiplier algebra $M(\K)$ is the $C^*$-algebra $B(\ell^2(\N))$
of bounded linear operators on $\ell^2(\N)$.
We denote by $\C$ the $C^*$-subalgebra of $\K$ generated by the diagonal operators
$e_{i,i}, i \in \N$.
Since the graph $G_Z =(\V_Z,\E_Z)$ is bipartite,
we have
$\E_Z = \E_C\cup \E_D$ and
the vertex set $\V_Z$ is decomposed into $\V_{C,D}\cup\V_{D,C}$
such that 
\begin{align*}
\V_{C,D} &= \{ I \in \V_Z \mid t(c) =I \text{ for some } c \in \E_C\}, \\
\V_{D,C} &= \{ I \in \V_Z \mid t(d) =I \text{ for some } d \in \E_D\}.
\end{align*} 
In what follows, we denote 
by $\E$ the edge set $\E_Z$,
and
by $\V$ the vertex set $\V_Z$,
respectively.
For a vertex $I\in \V$,
let us denote by $\E^I$ (resp. $\E_I$) the set of edges in $\E$ 
whose terminals (resp. sources) are $I$, that is,
\begin{equation*}
\E^I = \{ e \in \E \mid t(e) = I\}, \qquad
\E_I = \{ e \in \E \mid s(e) = I\}.
\end{equation*}
\begin{lemma}\label{lem:isomet}
For a fixed $I \in \V_{C,D}$, we may assign a family $s_c, c \in \E^I $ of isometries on the Hilbert space $\ell^2(\N)$ such that 
\begin{equation}
\sum_{c \in \E^I}s_c s_c^* = 1,
\qquad
s_c^* s_c = 1
\quad
\text{ and }
\quad
s_c \C s_c^* \subset \C,
\qquad
s_c^* \C s_c \subset \C
\quad
\text{ for } c \in \E^I. \label{eq:sc}
\end{equation}
\end{lemma}
\begin{proof}
Suppose that
$\E^I = \{c_1,\cdots,c_k\}$.
If $k=1$, one may take $s_c =1$.
Suppose $k \ge 2$.
Let $\delta_j, j \in \N$ be the complete orthonormal basis of $\ell^2(\N)$ 
defined by \eqref{eq:delta}.
Define operators $s_{c_i}$ on $\ell^2(\N)$ by
setting
\begin{equation}
s_{c_i}\delta_j = \delta_{k(j-1) +i }, \qquad i=1,\dots,k, \,\,  j\in \N.
\end{equation}
It is easy to see that the operators $s_{c_i}$ on $\ell^2(\N)$
have the desired properties. 
\end{proof}
For each vertex $I \in \V_{C,D}$,
take a family of isometries
$s^I_c, c \in \E^I$ having the properties \eqref{eq:sc}
and put 
\begin{equation*}
V^I_c = S_c \otimes s_c^{I*}\quad
\text{ in }
\quad
\OZ\otimes B(\ell^2(\N))
\quad
\text{ for }
 \quad c \in \E^I
\end{equation*}
and define the operator $V$ by setting
\begin{equation}
V = \sum_{I \in \V_{C,D}} \sum_{c \in \E^I} V_c^I
 (= \sum_{c \in \E_C} S_c \otimes s_c^{I*}) \label{eq:V}
\end{equation}
which belongs to $\OZ \otimes B(\ell^2(\N))$.
We note that 
for $c, c' \in \E_C$, 
the operator $S_c S_{c'}^* \ne 0$ if and only if 
$S_c^*S_c S_{c'}^*S_{c'} \ne 0.$
As the latter condition is equivalent to the condition that 
$t(c) = t(c')$, we see that 
 $S_c S_{c'}^* \ne 0$ if and only if $c, c' \in \E^I$ for some $I \in \V.$
We also notice that 
if $c$ belongs to
$\E^I$, then $S_c = \sum_{d \in \E_I}S_c S_d S_d^*$, so that 
the identity
$V_c =\sum_{d \in \E_I}S_c S_d S_d^* \otimes s_c^{I*}$
holds.
We then have the following lemma.
\begin{lemma}
The partial isometry
$V \in \OZ\otimes B(\ell^2(\N))$
defined above 
has the following properties:
\begin{gather}
V V^* = P_A \otimes 1, \qquad
V^* V = P_B \otimes 1, \qquad
V \C V^* \subset \C,
\qquad
V^* \C V \subset \C
\quad
\text{ and }\\
(\rho^Z_t \otimes \id)(V) =e^{2 \pi \sqrt{-1} t} V
\quad
\text{ for }
\quad
 t \in \R/\Z.
\end{gather}
\end{lemma}
\begin{proof}
 We have the following equalities:
\begin{align*}
VV^* 
&= \sum_{I \in \V_{C,D}} (\sum_{c \in \E^I}S_c \otimes s_c^{I*})\cdot
     \sum_{I' \in \V_{C,D}} (\sum_{c' \in \E^{I'}}S_{c'}^* \otimes s_{c'}^{I'}) \\
&= \sum_{I, I' \in \V_{C,D}} \sum_{c \in \E^I}\sum_{c' \in \E^{I'}}
     S_c S_{c'}^* \otimes s_c^{I*} s_{c'}^{I'} \\
&= \sum_{I \in \V_{C,D}} \sum_{c, c' \in \E^I} S_c S_{c'}^* \otimes s_c^{I*} s_{c'}^{I} \\
&= \sum_{I \in \V_{C,D}} \sum_{c \in \E^I} S_c S_{c}^* \otimes s_c^{I*} s_{c}^I \\
&= \sum_{ c \in \E_C} S_c S_{c}^* \otimes 1 = P_A\otimes 1.
\end{align*}
We also have 
\begin{align*}
V^* V 
&= \sum_{I \in \V_{C,D}} (\sum_{c \in \E^I}S_c^* \otimes s_c^I)\cdot
     \sum_{I' \in \V_{C,D}} (\sum_{c' \in \E^{I'}}S_{c'} \otimes s_{c'}^{I'*}) \\
&= \sum_{I, I' \in \V_{C,D}} \sum_{c \in \E^I}\sum_{c' \in \E^{I'}}
     S_c^* S_{c'} \otimes s_c^I s_{c'}^{I'*} \\
&= \sum_{I \in \V_{C,D}} \sum_{c \in \E^I} S_c^* S_{c} \otimes s_c^I s_{c}^{I*} \\
&= \sum_{I \in \V_{C,D}} \sum_{c \in \E^I} 
(\sum_{d \in \E_I} S_d S_{d}^*) \otimes s_c^I s_{c}^{I*} \\
&= \sum_{I \in \V_{C,D}} \sum_{d \in \E_I}( S_d S_{d}^* \otimes 
  \sum_{c \in \E^I} s_c^I s_{c}^{I*} )\\
&= \sum_{I \in \V_{C,D}} \sum_{d \in \E_I}( S_d S_{d}^* \otimes 1)\\
&= \sum_{d \in \E_D}  (S_d S_{d}^* \otimes 1) = P_B\otimes 1.
\end{align*}
Since for $a \in \C$ and $c \in \E^I, c' \in \E^{I'}$,
we see
$V_c^I a V_{c'}^{I'*} =0$ if $c\ne c'$ 
and 
$V_c^{I*} a V_{c'}^{I'} =0$ if $c\ne c'$,
so that we have
\begin{equation*}
V a V^*  = \sum_{I \in V_{C,D}}\sum_{c \in \E_C}V_c^I a V_c^{I*}, \qquad 
V^* a V  = \sum_{I \in V_{C,D}}\sum_{c \in \E_C}V_c^{I*} a V_c^I. 
\end{equation*}
It is easy to see that 
both  elements 
$V_c^I a V_c^{I*}$ and $V_c^{I*} a V_c^I$
belong to $\C$
so that we have 
$
V \C V^* \subset \C,
$
and
$
V^* \C V \subset \C.
$
The equality
$(\rho^Z_t \otimes \id)(V) =e^{2 \pi \sqrt{-1} t} V$
 for
$
 t \in \R/\Z
$ is clear
because 
$
(\rho_t^Z\otimes\id)(S_c \otimes s_c^{I*}) = 
e^{2 \pi \sqrt{-1} t}(S_c \otimes s_c^{I*}).
$
\end{proof}


As in \cite[Section 4]{MaPre2015},
The homomorphisms  
$\phi:C(X_A,\Z) \rightarrow  C(X_B,\Z)$
and
$\psi:C(X_B,\Z) \rightarrow  C(X_A,\Z)$
defined by
\begin{equation*}
\phi(f) = \sum_{d \in \E_D}S_d f S_d^*,
\qquad
\psi(g) = \sum_{c \in \E_C}S_c g S_c^*
\end{equation*}
for 
$f \in C(X_A,\Z), g \in C(X_B,\Z)$
 satisfy 
\eqref{eq:psiphi}.
We will show the following theorem.
\begin{theorem}\label{thm:main1}
Let $A, B$ be nonnegative square matrices 
both of which are irreducible and not any permutations.  
Suppose that they are elementary equivalent such that
$A = CD$ and $B =DC$ for some nonnegative rectangular matrices $C$ and $D$.
Then there exists an isomorphism
$\Phi:\SOA \rightarrow \SOB$ satisfying 
$\Phi(\SDA) = \SDB$ such that
\begin{equation}
\Phi  \circ (\rho^{A,\psi(g)}_t\otimes \id)
 = (\rho^{B,g}_t\otimes\id) \circ \Phi
\quad
\text{ for }
g \in C(X_B,\Z), \, 
 t \in \T. \label{eq:Phirho1}
\end{equation}
In particular, we have
\begin{equation}
\Phi  \circ (\rho^{A}_t\otimes \id)
 = (\rho^{B}_t\otimes\id) \circ \Phi
\quad
\text{ for }
 t \in \T. \label{eq:Phirho}
\end{equation}
\end{theorem}
\begin{proof}
Through the identification \eqref{eq:4.1},
the restriction of the map
$x \otimes T \in \OZ \otimes\K 
\longrightarrow 
 V^*(x\otimes T)V
\in
\OZ \otimes\K$
to $P_A\OZ P_A \otimes \K$ yields an isomorphism
from
$\OA\otimes\K$ to $\OB\otimes\K$,
because
$V$ belongs to
$\OZ\otimes B(\ell^2(\N))$ and
$\OZ\otimes B(\ell^2(\N)) \subset M(\OZ\otimes\K)$,
 the multiplier algebra of $\OZ\otimes\K$.
The isomorphism is denoted by $\Phi$.
The previous lemma ensures us that $\Phi$ satisfies 
$\Phi(\SDA) = \SDB$ and the identity \eqref{eq:Phirho}.
We will show the equality \eqref{eq:Phirho1}.
We write 
$V = \sum_{c \in \E_C}S_c\otimes s_c^*$
instead of
$\sum_{I \in \V_{C,D}}\sum_{c \in \E^I} S_c \otimes s_c^{I*}.
$
For $g \in C(X_B,\Z), a_i \in \E_A, T\in \K$, we have the equalities.
\begin{align*}
\Phi  \circ (\rho^{A,\psi(g)}_t\otimes \id)(S_{a_i}\otimes T)
=&V^* (\rho^{A,\psi(g)}_t(S_{a_i})\otimes T) V\\
=&(\sum_{c \in \E_C} S_c^* \otimes s_c)
  (U_t(\psi(g))S_{a_i}\otimes T) (\sum_{c' \in \E_C} S_{c'} \otimes s_{c'}^*)\\
=& \sum_{c, c' \in \E_C} U_t(S_c^*\psi(g)S_c) S_c^*S_{a_i}S_{c'} 
\otimes s_c T s_{c'}^* \\
=& \sum_{c,c' \in \E_C} U_t(g) S_c^*S_{a_i}S_{c'} 
\otimes s_c T s_{c'}^* \\
=& \sum_{c' \in \E_C} U_t(g) S_{c(a_i)}^*S_{c(a_i)}S_{d(a_i)}S_{c'} 
\otimes s_{c(a_i)} T s_{c'}^* \\
=& \sum_{c' \in \E_C}  S_{c(a_i)}^*S_{c(a_i)}U_t(g) S_{d(a_i)}S_{c'} 
\otimes s_{c(a_i)} T s_{c'}^* \\
=& \sum_{c' \in \E_C}  S_{c(a_i)}^*S_{c(a_i)} \rho_t^{B,g}(S_{d(a_i)}S_{c'}) 
\otimes s_{c(a_i)} T s_{c'}^* \\
=& (\rho_t^{B,g} \otimes \id)(
   \sum_{c' \in \E_C}  S_{c(a_i)}^*S_{c(a_i)} S_{d(a_i)}S_{c'}) 
\otimes s_{c(a_i)} T s_{c'}^* ) \\
=& (\rho_t^{B,g} \otimes \id)(
\sum_{c,c' \in \E_C} S_c^*S_{a_i}S_{c'} 
\otimes s_c T s_{c'}^* ) \\
=& (\rho^{B,g}_t\otimes\id)
(V^* (S_{a_i}\otimes T) V)\\
=& ((\rho^{B,g}_t\otimes\id) \circ \Phi)(S_{a_i}\otimes T).
\end{align*}
Therefore we have the equality
\eqref{eq:Phirho1}.
\end{proof}
\begin{remark}
{\bf 1.}  Theorem \ref{thm:main1} can be directly seen
from Carlsen-Rout's results \cite[Theorem 3.3 and Theorem 5.1]{CR}.
However our proof above is completely different from theirs,
and also our construction of the isomorphism
$\Phi$ will be used in order to clarify its K-theoretic behavior in the following section. 

{\bf 2.} 
In the author's earlier paper \cite{MaPre2016}, 
a similar result \cite[Theorem 1.3]{MaPre2016} to Theorem \ref{thm:main1}
has been obtained. 
The previous one \cite[Theorem 1.3]{MaPre2016}
asserted cocycle conjugacy of gauge actions $\rho^A$ and $\rho^B$
between $\SOA$ and $\SOB$,
whereas our result Theorem \ref{thm:main1} 
does not need the cocycles to obtain the equalities \eqref{eq:Phirho1}, \eqref{eq:Phirho}.

{\bf 3.} Let $\O2$ be the Cuntz algebra of order $2$
which is the Cuntz--Krieger algebra for the matrix $[2]$.
Let $\D2$ be the canonical maximal abelian $C^*$-subalgebra $\mathcal{D}_{[2]}$
of $\O2$. 
We fix the pair $\O2$ and $\D2$.
We keep the situation $A = CD, B=DC$.
For each $I \in \V_{C,D}$,
the set $\E^I$ of edges terminating at the vertex $I$ is finite,
we may find a family $\{t_c \}_{c \in \E^I}$
of isometries in $\O2$ such hat 
\begin{equation}
\sum_{c \in \E^I}t_c t_c^* = 1,
\qquad
t_c^* t_c = 1
\quad
\text{ and }
\quad
t_c \D2 t_c^* \subset \D2,
\qquad
t_c^* \D2 t_c \subset \D2
\quad
\text{ for } c \in \E^I 
\end{equation}
which are the same relations as \eqref{eq:sc}.
Put
$V_2 =\sum_{I \in \V_{C,D}}\sum_{c \in \E^I} S_c \otimes t_c^{*}$ in 
$\OZ\otimes\O2$
instead of $V$ defined in \eqref{eq:V}.
We set
$\Phi_2 =\Ad(V_2): \OZ\otimes\O2 \longrightarrow \OZ\otimes\O2$.
By  a completely similar manner to the above discussions, we can show that 
there exists an isomorphism
$\Phi:\OA\otimes\O2 \rightarrow \OB\otimes\O2$ satisfying 
$\Phi(\DA\otimes\D2) = \DA\otimes\D2$ such that
\begin{equation*}
\Phi  \circ (\rho^{A}_t\otimes \id)
 = (\rho^{B}_t\otimes\id) \circ \Phi
\quad
\text{ for }
 t \in \T.
\end{equation*}
Hence we have the following proposition.
\begin{proposition}
Suppose that irreducible nonnegative matrices  $A$ and $B$ are strong shift equivalent.
Then there exists an isomorphism
$\Phi:\OA\otimes\O2 \rightarrow \OB\otimes\O2$ satisfying 
$\Phi(\DA\otimes\O2) = \DB\otimes\D2$ such that
\begin{equation*}
\Phi  \circ (\rho^{A}_t\otimes \id)
 = (\rho^{B}_t\otimes\id) \circ \Phi
\quad
\text{ for }
 t \in \T
\end{equation*}
\end{proposition}
Since both the algebras $\OA, \OB$ are unital, simple, purely infinite and nuclear,
 by \cite{Ro}, the above  $C^*$-algebras 
$\OA\otimes\O2, \OB\otimes\O2
$ are isomorphic to $\O2$,
and 
$\DA\otimes\D2, \DB\otimes\D2$ are maximal abelian
in 
$\OA\otimes\O2, \OB\otimes\O2,
$  respectively.
Hence the classification problem of two-sided topological Markov shifts
are closely related to that of circle actions on the Cuntz algebra $\O2$
trivially acting on its maximal abelian $C^*$-subalgebras. 
\end{remark}

\section{Isomorphism $\Phi_*: K_0(\OA)\longrightarrow K_0(\OB)$}
Let $A, B$ be nonnegative square matrices 
both of which are irreducible and not any permutations.  
Suppose that they are elementary equivalent such that
$A = CD$ and $B =DC$ for some nonnegative rectangular matrices $C, D$.
By Theorem \ref{thm:main1},
we have an isomorphism
$\Phi:\SOA \rightarrow \SOB$ satisfying 
$\Phi(\SDA) = \SDB$
and \eqref{eq:Phirho1}.
We will in this section  clarify the 
K-theoretic behavior 
$\Phi_* : K_0(\OA) \rightarrow K_0(\OB)$
of the isomorphism 
$\Phi:\SOA \rightarrow \SOB$.
As in the preceding section, 
for the
$N\times N$ matrix 
$A =[A(i,j)]_{i,j=1}^N $,
we have a directed graph
$G_A = (\V_A,\E_A)$
and its transition matrix
$A^G =[A^G(i,j)]_{i,j=1}^{N_A}$ with entries in $\{0,1\}$.
For the other matrix $B$, 
we similarly have 
a directed graph
$G_B = (\V_B, \E_B)$ 
and its transition matrix
$B^G =[B^G(i,j)]_{i,j=1}^{M_B}$ with entries in $\{0,1\}$.
Let us denote their vertex sets and edge sets 
by
$\V_A =\{ I^A_1, \dots, I_N^A\}$,  
$\V_B =\{ I_1^B,\dots,I_M^{B} \}$
and
$\E_A =\{a_1,\dots,a_{N_A} \}$,
 $\E_B = \{ b_1, \dots,b_{M_B}\}$
respectively.
Recall that 
the Cuntz-Krieger algebras $\OA$ and $\OB$ 
are defined as the Cuntz--Krieger algebras
 ${\mathcal{O}}_{A^G}$ and ${\mathcal{O}}_{B^G}$, respectively.


As in Section 2, 
for any $a_i \in \E_A$, 
there  exist
$c(a_i) \in \E_C$ and $d(a_i) \in \E_D$ such that 
$a_i$ is written $c(a_i)d(a_i)$. 
Similarly  for any edge $b_l \in \E_B$,
there exist  
$d(b_l) \in \E_D$ and $c(b_l) \in \E_C$ 
such that
$b_l$ is written $d(b_l)c(b_l).$
The $N_A\times M_B$ matrix 
$\hat{D} = [\hat{D}(i,l)]_{i=1,\dots,N_A}^{l=1,\dots,M_B}$ 
has been defined in \cite{MaPre2016}
by
\begin{equation}
\hat{D}(i,l) =
\begin{cases} 
1 &  \text{  if  } d(a_i) = d(b_l), \\
0 &  \text{  otherwise.}
\end{cases} \label{eq:Dhat}
\end{equation}
It is direct or 
due to  \cite[Lemma 3.1]{MaPre2016}
to see that the multiplication 
of the transposed matrix 
$\hat{D}^t:$ 
$
[n_i]_{i=1}^{N_A} \in \Z^{N_A}\rightarrow 
[\sum_{i=1}^{N_A} \hat{D}(i,l) n_i]_{l=1}^{M_B} \in\Z^{M_B}
$ 
induces a homomorphism from
$\Z^{N_A}/{(\id - {(A^G)}^{t})\Z^{N_A}}$
to 
$\Z^{M_B}/{(\id - {(B^G)}^{t})\Z^{M_B}}$
as abelian groups, which is denoted by 
$\Phi_{\hat{D}^t}$.

Let us denote by
$e_i = (0,\dots,0,\overset{i}{1},0,\dots,0)$
the vector in 
$\Z^{N_A}$
whose $i$th component is $1$, elsewhere  zero. 
Its class 
in $\Z^{N_A}/{(\id - {(A^G)}^{t})\Z^{N_A}} $ 
is denoted by
$[e_i]$.
J. Cuntz in \cite{Cu3}  showed that
the map
$\epsilon_{A^G}: K_0({\mathcal{O}}_{A^G}) 
\rightarrow \Z^{N_A}/{(\id - {(A^G)}^{t})\Z^{N_A}}$
defined by $\epsilon_{A^G}([S_{a_i}S_{a_i}^*] ) = [e_i]$
yields an isomorphism of abelian groups. 
We are assuming that  $A=CD, B=DC.$
Let  
$\Phi:\SOA \rightarrow \SOB$ 
be the isomorphism defined in Theorem \ref{thm:main1}.
We show the following proposition.
\begin{proposition}\label{prop:KTD}
Let  $\Phi_*: K_0({\mathcal{O}}_{A^G}) \longrightarrow K_0({\mathcal{O}}_{B^G})
$ be the induced isomorphism from
$\Phi:\SOA \rightarrow \SOB$.
Then we have 
$\Phi_*\circ \epsilon_{A^G} =\epsilon_{B^G}\circ \Phi_*.$
\end{proposition}
\begin{proof}
Let $p_1$ be the rank one projection onto the vector 
$\delta_1 \in \ell^2(\N)$.
By \cite{Cu3}, the $K_0$-group 
$K_0(\OAG)$ of $\OAG$ is generated by the projections of the form
$$
S_{a_i}S_{a_i}^*\otimes p_1 \in \OAG\otimes\K, \qquad i=1,\dots,N_A.
$$
It then follows that 
\begin{align*}
\Phi(S_{a_i}S_{a_i}^* \otimes p_1)
& = (\sum_{c \in \E_C}S_c\otimes s_c^*)^*(S_{a_i}S_{a_i}^* \otimes p_1)  
     (\sum_{c' \in \E_C}S_{c'}\otimes s_{c'}^*) \\
& = \sum_{c, c' \in \E_C} S_c^* S_{a_i}S_{a_i}^*S_{c'}
      \otimes s_c p_1 s_{c'}^*. 
\end{align*}
We know that 
$S_c^* S_{a_i}S_{a_i}^*S_{c'} =0$ if $c \ne c'$.
We also know that 
$S_c^* S_{a_i} =0$ if $c \ne c(a_i)$.
Hence the above last terms go to
the following
\begin{equation*} 
S_{c(a_i)}^* S_{a_i}S_{a_i}^*S_{c(a_i)}
      \otimes s_{c(a_i)} p_1 s_{c(a_i)}^*
=S_{d(a_i)} S_{d(a_i)}^*
      \otimes s_{c(a_i)} p_1 s_{c(a_i)}^*.
\end{equation*}
The projection 
$S_{d(a_i)} S_{d(a_i)}^*$ belongs to $\OBG$.
In the $K_0$-group of $\OBG$, we have 
$[S_{d(a_i)} S_{d(a_i)}^*\otimes s_{c(a_i)} p_1 s_{c(a_i)}^*]
=[S_{d(a_i)} S_{d(a_i)}^*\otimes  p_1]
$ in $K_0(\OBG)$.
By \cite[Lemma 3.4]{MaPre2016}, we see that
\begin{equation}
 S_{d(a_i)}S_{d(a_i)}^* = \sum_{l=1}^{M_B} \hat{D}(i,l) S_{b_l}S_{b_l}^*
\end{equation}
 so that we have in $K_0(\OBG)$ 
\begin{align*}
\Phi_*([S_{a_i}S_{a_i}^* \otimes p_1])
& = [\Phi(S_{a_i}S_{a_i}^* \otimes p_1)] \\
& = [S_{d(a_i)} S_{d(a_i)}^*\otimes s_{c(a_i)} p_1 s_{c(a_i)}^*] \\
& =[S_{d(a_i)} S_{d(a_i)}^*\otimes  p_1] \\
&=\sum_{l=1}^{M_B} \hat{D}(i,l)[ S_{b_l}S_{b_l}^*]
\end{align*}
This shows that $\Phi_*\circ \epsilon_{A^G} =\epsilon_{B^G}\circ \Phi_*.$
\end{proof}

We will define the matrices  $R_A$ and $S_A$
to connect between $A$ and $A^G$.
They are  the $N\times N_A$ matrix and $N_A\times N$ matrix 
 defined by 
\begin{equation*}
R_A(j,i) =
\begin{cases}
1 & \text{ if }   I_j^A = s(a_i),\\
0 & \text{ otherwise, }
\end{cases}
\qquad
S_A(i,j) =
\begin{cases}
1 & \text{ if }  t(a_i) = I_j^A,\\
0 & \text{ otherwise, }
\end{cases}
\end{equation*}
for $i = 1,\dots, N_A$ and $j=1,\dots,N,$ respectively.
It is direct to see that 
$A = R_A S_A$ and 
$A^G = S_A R_A$.
The matrices 
$R_B, S_B$ for the other matrix $B$
are similarly defined such that 
$B = R_B S_B$ and 
$B^G = S_B R_B$.
There are 
natural homomorphisms
\begin{gather*}
\Phi_{S_A^t}: \Z^{N_A}/{(\id - {(A^G)}^{t})\Z^{N_A}}
\rightarrow
\Z^{N}/{(\id - {A}^{t})\Z^{N}}, \\
\Phi_{S_B^t}:\Z^{M_B}/{(\id - {(B^G)}^{t})\Z^{M_B}} 
\rightarrow
\Z^{M}/{(\id - {B}^t)\Z^{M}}
\end{gather*}
of abelian groups 
induced from the matrix $S_A^t: \Z^{N_A}\rightarrow \Z^{N_A}$ 
and
$S_B^t: \Z^{M_B}\rightarrow \Z^{M_B}$,
respectively.
The homomorphisms
$\Phi_{S_A^t}$ and $\Phi_{S_B^t}
$
are both isomorphisms
because
their inverses are  given by the homomorphisms induced by
$R_A^t$ and $R_B^t$, respectively.
Since the conditions 
$A=CD, B=DC$ imply that  
$AC =CB$ and hence $C^t A^t = B^t C^t$,
we see the matrix $C^t: \Z^{N}\rightarrow \Z^{M}$ 
induces a homomorphism 
\begin{equation*}
\Phi_{C^t}: \Z^{N}/{(\id - {A}^{t})\Z^{N}} \longrightarrow \Z^{M}/{(\id - {B}^{t})\Z^{M}}
\end{equation*}
of abelian groups, 
which is  actually an isomorphism having $\Phi_{D^t}$ as its inverse.
The following lemma is seen in \cite{MaPre2016}. 
Hiroki Matui kindly pointed out the second assertion (ii). 
The author would like to thank him.
\begin{lemma}[{\cite[Lemma 3.5]{MaPre2016}}] \label{lem:Matuilemma}
\hspace{6cm}
\begin{enumerate}
\renewcommand{\theenumi}{\roman{enumi}}
\renewcommand{\labelenumi}{\textup{(\theenumi)}}
\item $\Phi_{S^t_{B}}\circ \Phi_{\hat{D}^t} = \Phi_{C^t}\circ\Phi_{S^t_{A}}.$
\item
$\Phi_{S^t_A}([(1,1,\dots,1)]) = [(1,1,\dots,1)].$
\end{enumerate}
\end{lemma}
%
Let us denote by 
$\epsilon_A $
the isomorphism 
$ \Phi_{S_A^t} \circ \epsilon_{A^G}: 
K_0(\OA) \rightarrow  \Z^{N}/{(\id - {A}^{t})\Z^{N}}$,
which is an isomorphism satisfying 
$\epsilon_A([1_A]) =[(1,1,\dots,1)].$ 
We thus obtain the following theorem:
\begin{theorem}\label{thm:KC}
Let $A, B$ be nonnegative square matrices 
both of which are irreducible and not any permutations.  
Suppose that they are elementary equivalent such that 
$A= CD, B= DC$ for some nonnegative rectangular matrices $C, D$.
Then the diagram
$$
\begin{CD}
K_0({\OA}) @>\Phi_* >> K_0({\OB}) \\
@V{\epsilon_{A} }VV  @VV{\epsilon_{B}}V \\
\Z^{N}/{(\id - {A}^{t})\Z^{N}} @> \Phi_{C^t} >> \Z^{M}/{(\id - {B}^t)\Z^{M}} 
\end{CD}
$$
 of isomorphisms is commutative.
\end{theorem}
Recall two nonnegative matrices 
$A, B$ are said to be strong shift equivalent if they are connected 
by a finite chain of elementary equivalences such as
\eqref{eq:SSE}.
Then we have an isomorphism 
$
\Phi_{{(C_1 C_2 \cdots C_n)}^t}:
\Z^{N}/{(\id - {A}^{t})\Z^{N}} 
\longrightarrow
 \Z^{M}/{(\id - {B}^t)\Z^{M}} 
$
which is induced  by  the left multiplication
of the matrix 
$
{(C_1 C_2 \cdots C_n)}^t:
\Z^{N} 
\longrightarrow
 \Z^{M} 
$
whose  inverse is given by 
 $\Phi_{{(D_n \cdots D_2 D_1)}^t}:
 \Z^{M}/{(\id - {B}^t)\Z^{M}}\rightarrow 
\Z^{N}/{(\id - {A}^{t})\Z^{N}}.
$
We thus have the following corollary.
\begin{corollary}\label{cor:KC}
Suppose that
nonnegative irreducible  matrices $A, B$ are strong shift equivalent in $n$-step
such that 
$
A \underset{C_1,D_1}{\approx}\cdots \underset{C_{n},D_{n}}{\approx}B 
$
for some  nonnegative rectangular matrices
$ C_1, \dots, C_n$ 
and 
$D_1, \dots,D_n$.
Then there exists an isomorphism
$\Phi:\SOA \rightarrow \SOB $ of $C^*$-algebras such that
\begin{equation*}
\Phi(\SDA) = \SDB, \qquad
\Phi \circ (\rho^{A}_t \otimes\id) = (\rho^{B}_t \otimes \id) \circ \Phi,
\end{equation*}
and the following diagram of isomorphisms is commutative
$$
\begin{CD}
K_0({\OA}) @>\Phi_* >> K_0({\OB}) \\
@V{\epsilon_{A} }VV  @VV{\epsilon_{B}}V \\
\Z^{N}/{(\id - {A}^{t})\Z^{N}} 
@> \Phi_{{(C_1 C_2 \cdots C_n)}^t}>> \Z^{M}/{(\id - {B}^t)\Z^{M}} 
\end{CD}.
$$
\end{corollary}

\section{State splitting}

State splitting is a key procedure for 
constructing new directed graphs from an original directed graph in dividing topological conjugacy of topological Markov shifts.
The transition matrices of state splitting graphs give rise to strong shift equivalent matrices.
In \cite{Williams}, 
Williams has actually proved his strong shift equivalence theorem by decomposing
given directed graphs into finite sequence of state splittings.
Let $G =(\V,\E)$ be a finite directed graph.
Each element $I$ of $\V$ is a vertex of $G$ which we call a state instead of vertex.
We have two kinds of sate splitting procedures. 
One is out-splitting and the other is in-splitting. 
The former uses a partition of out-going edges from states,
whereas the latter uses a partition of in-coming edges to states.  
We fix a finite directed graph $G =(\V,\E)$ for a while.
For a vertex $I\in \V$,
recall that  $\E^I$ (resp. $\E_I$) denotes the set of edges in $\E$ 
whose terminals (resp. sources) are $I$, that is
\begin{equation*}
\E^I = \{ e \in \E \mid t(e) = I\}, \qquad
\E_I = \{ e \in \E \mid s(e) = I\}.
\end{equation*}
We note that the graph $C^*$-algebras constructed from state splitting graphs  
have been studied 
by several authors (see \cite{Bates}, \cite{BP},  \cite{MaETDS2004}, \cite{MPT}, \cite{Tomforde}, etc.).
\subsection{Out-splitting}
  We will first  explain out-splitting graph from $G$
following \cite[Definition 2.4.3]{LM}
(see also \cite{Kitchens}).
For each state $I\in \V$, 
we take a partition
of $\E_I$ such as
$\E_I = \E_I^1\cup\E_I^2\cup \cdots \cup \E_I^{m(I)}$,
which are disjoint partition.
We denote by $\mathP_I$ the partition of $\E_I$.
The whole partition  $\{ \mathP_I\}_{I \in \V}$ is denoted by $\mathP$.
 We construct a new graph $G^{[\mathP]} =(\V^{[\mathP]},\E^{[\mathP]})$
such as
$\V^{[\mathP]} = \cup_{I\in \V} \{ I^1, I^2, \dots,I^{m(I)}\}$.
For $e \in \E_I$
 which belongs to $e \in \E_I^i$ for some $i = 1,2,\dots,m(I)$
such that $t(e) = J$,
 then define new edges
$e^1,e^2,\dots,e^{m(J)}$ such that
$s(e^j) = I^i, \,  t(e^j) = J^j$ for all $j=1,2,\dots,m(J)$.
The set of such edges is $\E^{[\mathP]}.$
The resulting directed graph
$G^{[\mathP]} =(\V^{[\mathP]},\E^{[\mathP]})$
is called the {\it out-split graph formed from $G$ using $\mathP$}.
We note that if $G$ is irreducible, then the graph 
$G^{[\mathP]}$ is still irreducible (\cite[Exercise 2.4.5]{LM}).
The converse procedure to get a graph $G$ from $G^{[\mathP]}$
is called {\it out-amalgamation} (see \cite{LM}, \cite{Kitchens}).

In the above procedure to construct the new graph $G^{[\mathP]}$,
a bipartite graph 
$\hat{G}^{[\mathP]} =(\hat{\V}^{[\mathP]},\hat{\E}^{[\mathP]})$
is associated to it as follows.
Let 
$\hat{\V}^{[\mathP]} = \V\cup\V^{[\mathP]}$.
Define two kinds of edges such that
for the partition
$\E_I = \E_I^1\cup\E_I^2\cup \cdots \cup \E_I^{m(I)}$
the edge $i^n$ is defined such that
$s(i^n) = I, \,  t(i^n) = I^n$ for all $n=1,2,\dots,m(I)$.
For an edge $e \in \E_I^n$ such that $t(e) =J\in \V$,
the edge 
$\bar{e}$ is defined such that 
$s(\bar{e}) = I^n,\, t(\bar{e}) = J$.
Then the set $\hat{\E}^{[\mathP]}$ consists of such edges, that is
\begin{equation}
\hat{\E}^{[\mathP]} =\{i^n \mid n=1,2,\dots,m(I), I \in \V\}
\cup \{ \bar{e} \mid e \in \E \}. \label{eq:hatEP}
\end{equation}
We define their transition matrices
$C^{[\mathP]}, D^{[\mathP]}$ by setting
\begin{equation*}
C^{[\mathP]}(I,J^k) = 
\begin{cases}
1 & \text{ if } I=J \\
0 & \text{ otherwise},
\end{cases}
\qquad
 D^{[\mathP]}(I^n,J) = |\E_I^n \cap \E^J|.
\end{equation*}
Let us denote by 
$A^{{[\mathP]}}$
the transition matrix of the graph  
$G^{[\mathP]}$.
Then we have
\begin{equation*}
A = C^{[\mathP]}D^{[\mathP]}, \qquad
A^{[\mathP]} = D^{[\mathP]}C^{[\mathP]}.
\end{equation*}
We set the matrix
\begin{equation*}
\AGHP
=
\begin{bmatrix}
0 & C^{{[\mathP]}} \\
D^{{[\mathP]}} & 0
\end{bmatrix}
\end{equation*}
which is the transition matrix of the 
bipartite graph $\hat{G}^{[\mathP]}$.
Let 
$\OAGHP$ be the Cuntz--Krieger algebra for the matrix $\AGHP$.
Let
\begin{equation*}
S_{i^n }, \quad S_{\bar{e}} \quad 
\text{ for }
n=1,2,\dots,m(I), I \in \V, \,
e \in \E
\end{equation*}
be the canonical generating partial isometries for the algebra $\OAGHP$
assigned by the edges \eqref{eq:hatEP}
in the bipartite graph $\hat{G}^{[\mathP]}$. 
We set projections $P_A, P_{\AGP} $ in $\OAGHP$
\begin{equation*}
P_A = \sum_{e \in \E}S_{\bar{e}}S_{\bar{e}}^*, \qquad
P_{\AGP} = \sum_{i^n} S_{i^n }S_{i^n }^*.
\end{equation*}
Since the graph $\hat{G}^{[\mathP]}$ is bipartite,
as in \eqref{eq:4.1}, we know that 
\begin{gather}
P_A \OAGHP P_A = \OA, \qquad P_{\AGP} \OAGHP P_{\AGP} = \OAGP,
\label{eq:PAPAGP1} \\
\DAGHP P_A = \DA, \qquad  \DAGHP P_{\AGP} = \DAGP. \label{eq:PAPAGP2}
\end{gather}
We set a partial isometry $V^{[\mathP]}$ in $\OAGHP$
\begin{equation}
V^{[\mathP]}= \sum_{i^n} S_{i^n }.
\end{equation}
We then have
\begin{lemma}
$V^{[\mathP]} V^{[\mathP]*} = P_{\AGP}$ and 
$V^{[\mathP]*} V^{[\mathP]} = P_A.
$
\end{lemma}
\begin{proof}
We have
\begin{equation*}
V^{[\mathP]} V^{[\mathP]*}
= \sum_{i^n, j^k} S_{i^n }S_{j^k}^*.
\end{equation*}
Since 
$S_{i^n }^*S_{i^n } \cdot S_{j^k}^*S_{j^k} =0$
if  $t(i^n) \ne t(j^k)$,
we know that 
$S_{i^n }S_{j^k}^* =0$ if $i^n \ne j^k$.
Hence we  have
\begin{equation*}
V^{[\mathP]} V^{[\mathP]*}
= \sum_{i^n} S_{i^n }S_{i^n}^*
= P_{\AGP}.
\end{equation*}
As $S_{i^n} =S_{i^n} S_{\bar{e}}S_{\bar{e}}^*$ 
for $t(i^n) = s(\bar{e})$,
we have
\begin{align*}
V^{[\mathP]*} V^{[\mathP]}
=& \sum_{t(i^n) = s(\bar{e})} \sum_{t(j^k) = s(\bar{f})}
S_{\bar{e}}S_{\bar{e}}^* S_{i^n}^* S_{j^k}S_{\bar{f}}S_{\bar{f}}^* \\
=& \sum_{t(i^n) = s(\bar{e})} 
S_{\bar{e}}S_{\bar{e}}^* S_{i^n}^* S_{i^n}S_{\bar{e}}S_{\bar{e}}^* \\
=& \sum_{\bar{e}} 
S_{\bar{e}}S_{\bar{e}}^*  =P_A.
\end{align*}
\end{proof}
Hence we obtain
\begin{theorem}\label{thm:mainout}
Let $A$ be an irreducible nonnegative matrix and $G_A$ its directed graph.
Suppose that $\AGP$is the transition matrix of  the out-split graph of $G_A$ 
by a partition $\mathP$ of out-going edges of $G_A$.
Then there exists an isomorphism
$\Phi^{[\mathP]}:\OA \rightarrow \OAGP$ satisfying 
$\Phi^{[\mathP]}(\DA) = \DAGP$ such that
\begin{equation}
\Phi^{[\mathP]}  \circ \rho^{A}_t
 = \rho^{\AGP}_t \circ \Phi^{[\mathP]}
\quad
\text{ for } 
 t \in \T. \label{eq:PhirhoOut}
\end{equation}
\end{theorem}
\begin{proof}
Through the identifications
\eqref{eq:PAPAGP1} and 
\eqref{eq:PAPAGP2},
the restriction 
to $P_A\OAGHP P_A $
of the map
$x \in \OAGHP 
\longrightarrow 
V^{[\mathP]} x V^{[\mathP]*} \in
\OAGHP$
 yields an isomorphism
from
$\OA$ to $\OAGP$,
which we denote by $\Phi^{[\mathP]}$.
It is easy to see that
$V^{[\mathP]} \DA V^{[\mathP]*} = \DAGP$ so that 
$\Phi^{[\mathP]}(\DA) = \DAGP$.
As 
$\rho^{\AGHP}_t (V^{[\mathP]} ) = e^{2\pi\sqrt{-1}t} V^{[\mathP]}$
and
${\rho^{\AGHP}_t}|_{P_A\OAGHP P_A} = \rho^A_{2t}$
and
${\rho^{\AGHP}_t}|_{P_{\AGP}\OAGHP P_{\AGP}} = \rho^{\AGP}_{2t},$
we know that 
the equality
$
\Phi^{[\mathP]}  \circ \rho^{A}_t
 = \rho^{\AGP}_t \circ \Phi^{[\mathP]}.
$
\end{proof}
We note that Bates and Pask have shown that 
$\OA$ is isomorphic to $\OAGP$
(\cite[Theorem 3.2]{BP}).

\subsection{In-splitting}
  We will second explain in-splitting graph from $G$
following \cite[Definition 2.4.7]{LM}.
For each state $J\in \V$, 
we take a partition
of $\E^J$ such as
$\E^J = \E^J_1\cup\E^J_2\cup \cdots \cup \E^J_{m(J)}$,
which are disjoint partition.
We denote by $\mathP^J$ the partition of $\E^I$.
The whole partition  $\{ \mathP^J\}_{J \in \V}$ is denoted by $\mathP$.
 We construct a new graph 
$G_{[\mathP]} =(\V_{[\mathP]},\E_{[\mathP]})$
such as
$\V_{[\mathP]} = \cup_{J\in \V} \{ J_1, J_2, \dots, J_{m(J)}\}$.
For $e \in \E^J$
 which belongs to $e \in \E^J_j$ for some $j = 1,2,\dots,m(J)$
such that $s(e) = I$,
 then define new edges
$e_1,e_2,\dots,e_{m(J)}$ such that
$t(e_i) = J_j, \, s(e_i) = I_i$ for all $i=1,2,\dots,m(I)$.
The set of such edges is $\E_{[\mathP]}.$
The resulting directed graph
$G_{[\mathP]} =(\V_{[\mathP]},\E_{[\mathP]})$
is called the {\it in-split graph formed from $G$ using $\mathP$}.
We note that if $G$ is irreducible, then the graph 
$G_{[\mathP]}$ is still irreducible (\cite[Exercise 2.4.5]{LM}).
The converse procedure to get a graph $G$ from $G_{[\mathP]}$
is called {\it in-amalgamation} (see \cite{Kitchens}).

In the above procedure
to  construct the new graph $G_{[\mathP]}$,
we will construct a bipartite graph 
$\hat{G}_{[\mathP]} =(\hat{\V}_{[\mathP]},\hat{\E}_{[\mathP]})$
and its associated matrix in the following way.
Let 
$\hat{\V}_{[\mathP]} = \V\cup\V_{[\mathP]}$.
Define two kinds of edges such that
for the partition
$\E^J = \E^J_1\cup\E^J_2\cup \cdots \cup \E^J_{m(I)}$
the edge $j_n$ is defined such that
$t(j_n) = J, \, s(j_n) = J_n$ for all $n=1,2,\dots,m(J)$.
For an edge $e \in \E^J_n$ such that $s(e) =I\in \V$,
the edge 
$\hat{e}$ is defined such that 
$t(\hat{e}) = J_n, \,  s(\hat{e}) = I$.
Then the set $\hat{\E}_{[\mathP]}$ consists of such edges, that is
\begin{equation}
\hat{\E}_{[\mathP]} =\{j_n \mid n=1,2,\dots,m(J), J \in \V\}
\cup \{ \hat{e} \mid e \in \E \}. \label{eq:hatEIP}
\end{equation}
We define their transition matrices
$C_{[\mathP]}, D_{[\mathP]}$ by setting
\begin{equation}
C_{[\mathP]}(I,J_k) = 
\begin{cases}
1 & \text{ if } I=J \\
0 & \text{ otherwise},
\end{cases}
\qquad
 D_{[\mathP]}(J_n, I) = |\E^J_n \cap \E_I|.
\end{equation}
Let us denote by 
$A_{[\mathP]}$
the transition matrix of the graph  
$G_{[\mathP]}$.
Then we have
\begin{equation}
A = C_{[\mathP]}D_{[\mathP]}, \qquad
A_{[\mathP]} = D_{[\mathP]}C_{[\mathP]}.
\end{equation}
We set the matrix
\begin{equation}
\AGHIP
=
\begin{bmatrix}
0 & C_{[\mathP]} \\
D_{[\mathP]} & 0
\end{bmatrix}
\end{equation}
which is the transition matrix of the 
bipartite graph $\hat{G}_{[\mathP]}$.
Let 
$\OAGHIP$ be the Cuntz--Krieger algebra for the matrix $\AGHIP$.
Let
\begin{equation}
S_{j_n }, \quad S_{\hat{e}} \quad 
\text{ for }
n=1,2,\dots,m(J), J \in \V, \,
e \in \E
\end{equation}
be the canonical generating partial isometries for the algebra $\OAGIP$
assigned by the edges \eqref{eq:hatEIP}
in the bipartite graph $\hat{G}_{[\mathP]}$. 
We set projections $P_A, P_{\AGIP}$ in $\OAGHIP$
\begin{equation*}
P_A = \sum_{e \in \E}S_{\hat{e}}S_{\hat{e}}^*, \qquad
P_{\AGIP} = \sum_{j_n}S_{j_n }S_{j_n }^*.
\end{equation*}
Since the graph $\hat{G}_{[\mathP]}$ is bipartite,
as in \eqref{eq:4.1}, we know that
\begin{gather}
P_A \OAGHIP P_A = \OA, \qquad P_{\AGIP} \OAGHIP P_{\AGIP} = \OAGIP,
\label{eq:PAPAGIP1} \\
\DAGHIP P_A = \DA, \qquad  \DAGHIP P_{\AGIP} = \DAGIP. \label{eq:PAPAGIP2}
\end{gather}
For each $J \in \V$, as in Lemma \ref{lem:isomet},
we may assign a family 
$s_{j_n}, {j_n} \in \E^J $ of isometries on the Hilbert space $\ell^2(\N)$ such that 
\begin{equation*}
s_{j_n}^{*} s_{j_n} = 1, \quad
\sum_{{j_n} \in \E^J}s_{j_n} s_{j_n}^{*} = 1
\quad
\text{ and }
\quad
s_{j_n} \C s_{j_n}^{*} \subset \C,
\quad
s_{j_n}^{*} \C s_{j_n} \subset \C. 
\end{equation*} 
We set a partial isometry $V_{[\mathP]}$ 
in $\OAGHIP\otimes B(\ell^2(\N))$
\begin{equation}
V_{[\mathP]}= \sum_{j_n} S_{j_n }\otimes s_{j_n }^*.
\end{equation}
We have
\begin{lemma}
$
V_{[\mathP]}V_{[\mathP]}^* = P_{\AGIP}\otimes 1
$ 
and 
$V_{[\mathP]}^* V_{[\mathP]} = P_A\otimes 1.
$
\end{lemma}
\begin{proof}
We have
\begin{equation*}
V_{[\mathP]} V_{[\mathP]}^* 
= 
\sum_{j_n, j'_{n'}}
S_{j_n }S_{j'_{n'} }^*\otimes s_{j_n }^*s_{j'_{n'} }
\end{equation*}
Since 
$S_{j_n }^*S_{j_n } \cdot S_{j'_{n'}}^*S_{j'_{n'}} =0$
if 
$t(j) =t(j')$,
we see that 
$S_{j_n }S_{j'_{n'} }^* =0$ if  $j\ne j'$.
We also have 
$s_{j_n }^*s_{j_{n'}} \ne 0$
if and only if $n =n'$. 
Hence we have
\begin{equation*}
V_{[\mathP]}V_{[\mathP]}^* 
= 
\sum_{j_n}S_{j_n }S_{j_{n} }^*\otimes s_{j_n }^*s_{j_{n} }
= 
\sum_{j_n}S_{j_n }S_{j_{n} }^*\otimes 1
=P_{\AGIP}\otimes 1.
\end{equation*}
On the other hand, we have
\begin{equation*}
V_{[\mathP]}^* V_{[\mathP]} 
=
\sum_{j_n}\sum_{j'_{n'}}
S_{j_n }^*S_{j'_{n'} }\otimes s_{j_n }s_{j'_{n'} }^* 
= 
\sum_{j_n}
S_{j_n }^*S_{j_{n} }\otimes s_{j_n }s_{j_{n} }^*.
\end{equation*}
Since
$S_{j_n }^*S_{j_{n} } = S_{j'_{n'} }^*S_{j'_{n'} }$
if and only if 
$t(j_n) = t(j'_{n'}) =J$.
We may put
$Q_J =S_{j_n }^*S_{j_{n} } (= S_{j'_{n'} }^*S_{j'_{n'} })$.
Since
$\sum_{j_n \in \E^J}s_{j_n }s_{j_{n} }^*=1$,
we see that
$$
\sum_{j_n}
S_{j_n }^*S_{j_{n} }\otimes s_{j_n }s_{j_{n} }^*
 =\sum_{ J \in \V} Q_J\otimes 1
$$
As $Q_J =\sum_{e \in \E_J}S_{\hat{e}}S_{\hat{e}}^*$,
we have
$
V_{[\mathP]}^* V_{[\mathP]} 
 =  P_A \otimes 1.
$
\end{proof}
Hence we have
\begin{theorem}\label{thm:mainin}
Let $A$ be an irreducible nonnegative matrix and $G_A$ its directed graph.
Suppose that $\AGIP$is the transition matrix of the  in-split graph of $G_A$ 
by a partition $\mathP$ of in-coming edges of $G_A$.
Then there exists an isomorphism
$\Phi_{[\mathP]}:\OA\otimes\K \rightarrow \OAGIP\otimes\K$ satisfying 
$\Phi_{[\mathP]}(\DA\otimes\C) = \DAGIP\otimes\C$ such that
\begin{equation}
\Phi_{[\mathP]}  \circ (\rho^{A}_t\otimes\id)
 = (\rho^{\AGIP}_t \otimes\id)\circ \Phi_{[\mathP]}
\quad
\text{ for } 
 t \in \T. \label{eq:PhirhoIn}
\end{equation}
\end{theorem}
\begin{proof}
Through the identifications
\eqref{eq:PAPAGIP1} and 
\eqref{eq:PAPAGIP2},
the restriction 
to
to $P_A\OAGHIP P_A \otimes\K  $
of the map
$x \in \OAGHIP\otimes\K 
\longrightarrow  
V_{[\mathP]} x V_{[\mathP]}^*  \in
\OAGHIP\otimes\K$
 yields an isomorphism
from
$\SOA$ to $\OAGIP\otimes\K$,
which we denote by $\Phi_{[\mathP]}$.
It is easy to see that
$V_{[\mathP]}(\SDA) V_{[\mathP]}^*  = \DAGIP\otimes\C$ 
so that 
$\Phi_{[\mathP]}(\SDA) = \DAGIP\otimes\C$.
As 
$\rho^{\AGHIP}_t (V_{[\mathP]} ) = e^{2\pi\sqrt{-1}t}V_{[\mathP]}$
and
${\rho^{\AGHIP}_t}|_{P_A\OAGHIP P_A} = \rho^A_{2t}$
and
${\rho^{\AGHIP}_t}|_{P_{\AGIP}\OAGHIP P_{\AGIP}} = \rho^{\AGIP}_{2t}$,
we know that 
the equality
$
\Phi_{[\mathP]}  \circ (\rho^{A}_t\otimes\id)
 = (\rho^{\AGIP}_t\otimes\id) \circ \Phi_{[\mathP]}.
$
\end{proof}
We note that Bates and Pask have shown that 
$\OA\otimes\K$ is isomorphic to $\OAGIP\otimes\K$
(\cite[Theorem 5.3]{BP}).
\section{Transpose free isomorphic Cuntz--Krieger triplets}
We call the triplet
$(\OA,\DA,\rho^A)$
the Cuntz--Krieger triplet and write it 
$\CKT_A$.
Two 
 Cuntz--Krieger triplets
$\CKT_A$ and $\CKT_B$ are said to be isomorphic 
if there exists an isomorphism
$\Phi:\OA\longrightarrow \OB$
such that 
$\Phi(\DA) = \DB$
and
$\Phi\circ\rho^A_t = \rho^B_t\circ\Phi, \, t \in \T$.
It has been proved that 
$\CKT_A$ and $\CKT_B$ are isomorphic
 if and only if 
their underlying one-sided topological Markov shifts
$(X_A,\sigma_A)$ and
$(X_B,\sigma_B)$
are eventually conjugate
(\cite{MaJOT2015}). 
We denote by
$\bar{\CKT}_A$
the pair 
$(\CKT_{A^t}, \CKT_A)$
of the Cuntz--Krieger triplets.
\begin{definition} 
$\bar{\CKT}_A$
and
$\bar{\CKT}_B$
are said to be {\it transpose free isomorphic in }$1$-{\it step}
if $\CKT_A$ and $\CKT_B$ are isomorphic
or
$\CKT_{A^t}$ and $\CKT_{B^t}$ are isomorphic.
We write this situation as 
$\bar{\CKT}_A\underset{T-1}{\approx}\bar{\CKT}_B$.
$\bar{\CKT}_A$
and
$\bar{\CKT}_B$
are said to be {\it transpose free isomorphic in }$n$-{\it step}
or simply {\it transpose free isomorphic} if 
there exists a finite sequence $A_0, A_1,\dots,A_{n-1}, A_n$ of nonnegative irreducible square matrices, where $A = A_0, A_n =B$, 
such that  
\begin{equation}
\bar{\CKT}_A =
\bar{\CKT}_{A_0}\underset{T-1}{\approx}
\bar{\CKT}_{A_1}\underset{T-1}{\approx}
\cdots
\underset{T-1}{\approx}
\bar{\CKT}_{A_{n-1}}
\underset{T-1}{\approx}
\bar{\CKT}_{A_{n}} =
\bar{\CKT}_B. \label{eq:transiso}
\end{equation}
We write this situation as 
$\bar{\CKT}_A\underset{T-n}{\approx}\bar{\CKT}_B$
or simply
$\bar{\CKT}_A{\approx}\bar{\CKT}_B$.
\end{definition}
Then we have the following theorem.
\begin{theorem}\label{thm:transfree}
Suppose $A,B$ are nonnegative irreducible matrices
that are not any permutation matrices.
Then two-sided topological Markov shifts 
$(\bar{X}_A,\bar{\sigma}_A)$ and $(\bar{X}_B,\bar{\sigma}_B)$
are topologically conjugate 
if and only if
$\bar{\CKT}_A$
and
$\bar{\CKT}_B$
are transpose free isomorphic.
\end{theorem}
\begin{proof}
Suppose that
$(\bar{X}_A,\bar{\sigma}_A)$ and $(\bar{X}_B,\bar{\sigma}_B)$
are topologically conjugate. 
By  Williams' theorem \cite{Williams}, the underlying matrices
$A$ and $B$  are
 strong shift equivalent such that 
they are connected by a finite sequence of elementary equivalences
such as  
$A =A_0\approx A_1 \approx\cdots \approx A_{n-1} \approx A_n = B$
and their associated directed graphs
$G_{A_i}$ and $G_{A_{i+1}}$
are connected by one of the following four operations:

\hspace{17mm}out-splitting, \/ \/ out-amalgamation, \/ \/ in-splitting, \/\/ in-amalgamation. 

As the amalgamations are the converse operation of the splitting, we may consider 
only splittings.
If $G_{A_i}$ and $G_{A_{i+1}}$
are connected by out-splitting,
then by Theorem \ref{thm:mainout}, 
the Cuntz--Krieger triplets 
$\CKT_A$ and $\CKT_B$ are isomorphic.
If $G_{A_i}$ and $G_{A_{i+1}}$
are connected by in-splitting,
then 
their transposed graphs
$G_{A^t_i}$ and $G_{A^t_{i+1}}$
are connected by out-splitting,
so that 
the Cuntz--Krieger triplets 
$\CKT_{{A_i}^t}$ and $\CKT_{A_{i+1}^t}$ are isomorphic
and hence
$\bar{\CKT}_{A_i}$
and
$\bar{\CKT}_{A_{i+1}}$
are transpose free isomorphic in $1$-step.
Therefore we conclude that 
$\bar{\CKT}_A$
and
$\bar{\CKT}_B$
are transpose free isomorphic.

Conversely, 
suppose that
$\bar{\CKT}_A$
and
$\bar{\CKT}_B$
are transpose free isomorphic.
There exists a finite sequence $A_1,\dots,A_{n-1}$ 
of nonnegative irreducible square matrices
satisfying
\eqref{eq:transiso}
such that  
$\CKT_{{A_i}}$ and $\CKT_{{A_{i+1}}}$ are isomorphic,
or
the Cuntz--Krieger triplets 
$\CKT_{{A_i}^t}$ and $\CKT_{A_{i+1}^t}$ are isomorphic
for each $i=0,1,\dots,n-1$.
If the first case occurs,
their one-sided topological Markov shifts
$(X_{A_i},\sigma_{A_i})$ and
$(X_{A_{i+1}},\sigma_{A_{i+1}})$
are eventually conjugate,
so that 
their two-sided topological Markov shifts
$(\bar{X}_{A_i},\bar{\sigma}_{A_i})$ and $(\bar{X}_{A_{i+1}},\bar{\sigma}_{A_{i+1}})$
are topologically conjugate
(\cite{MaJOT2015}). 
If the second case occurs,
their one-sided topological Markov shifts
$(X_{A^t_i},\sigma_{A^t_i})$ and
$(X_{A^t_{i+1}},\sigma_{A^t_{i+1}})$
are eventually conjugate,
so that 
their two-sided topological Markov shifts
$(\bar{X}_{A^t_i},\bar{\sigma}_{A^t_i}) $ and $ (\bar{X}_{A^t_{i+1}},\bar{\sigma}_{A^t_{i+1}})$
are topologically conjugate.
Let $h: \bar{X}_{A^t_i}\longrightarrow \bar{X}_{A^t_{i+1}}$
be a homeomorphism whichi gives rise to a topological conjugacy between them.
Since the two-sided topological Markov shift
defined by the transposed matrix is 
the inverse of the original two-sided topological Markov shift,
the inverse $h^{-1}$
gives rise to a topological conjugacy between 
$
(\bar{X}_{A_i},\bar{\sigma}_{A_i})
$
and
$
 (\bar{X}_{A_{i+1}},\bar{\sigma}_{A_{i+1}}).
$
By connecting these topological conjugacies,
we have a topological conjugacy between 
$(\bar{X}_A,\bar{\sigma}_A), (\bar{X}_B,\bar{\sigma}_B)$.
\end{proof}




\bigskip

{\it Acknowledgments:}
The author thanks Hiroki Matui for his advices and suggestions
in Lemma \ref{lem:Matuilemma}.
A part of this work was done during staying at the Institut Mittag-Leffler
in the research program {\it Classification of operator algebras: complexity, rigidity and dynamics}.
The author is grateful to the organizers of the program and the Institut Mittag-Leffler for its hospitality
and support.
This work was also supported by JSPS KAKENHI Grant Number 15K04896.



\begin{thebibliography}{99}






\bibitem{AER}
{\sc S. Arklint, S. Eilers and E. Ruiz},
{\it A dynamical characterization of diagonal preserving $*$-isomorphisms of graph $C^*$-algebras},
arXiv:1605.01202 [mathOA]. 



\bibitem{Bates}
{\sc T. Bates},
{\it Application of the gauge-invariant uniqueness theorem for the Cuntz--Krieger algebras of directed graphs},
Bull. Austral. Math. Soc. 
{\bf 65}(2002), pp.  57--67.

\bibitem{BP}
{\sc T. Bates and D. Pask},
{\it Flow equivalence of graph algebras},
Ergodic Theory Dynam. Systems
{\bf 24}(2004), pp.  367--382.











\bibitem{Brown}
{\sc L. G. Brown},
{\it Stable isomorphism of hereditary subalgebras of $C^*$-algebras},
Pacific J.\ Math.\ {\bf 71}(1977), pp.\ 335--348.



\bibitem{BGR}
{\sc L. G. Brown, P. Green and M. A. Rieffel},
{\it Stable isomorphism and strong Morita equivalence of $C^*$-algebras},
 Pacific  J.\ Math.\ {\bf 71}(1977), pp.\ 349--363.

\bibitem{CR}
{\sc T. M. Carlsen and J. Rout},
{\it Diagonal-preserving gauge invariant isomorphisms of graph $C^*$-algebras},
preprint, arXiv: 1610.00692 [mathOA].


\bibitem{CRS}
{\sc T. M. Carlsen, E. Ruiz and A. Sims},
{\it Equivalence and stable isomorphism of groupoids, and diagonal-preserving 
stable isomorphisms of graph $C^*$-algebras and Leavitt path algebras},
preprint, arXiv: 1602.02602 [mathOA].














\bibitem{C}{\sc J. Cuntz},
{\it Simple $C^*$-algebras generated by isometries},
 Comm.\ Math.\ Phys.\ {\bf 57}(1977), pp.\ 173--185.











\bibitem{Cu3}
{\sc J. ~Cuntz}, 
{\it A class of $C^*$-algebras and topological Markov chains II: reducible chains and the Ext-functor for $C^*$-algebras},
Invent.\ Math.\ {\bf 63}(1980),
pp.\ 25--40.





\bibitem{CK}{\sc J. ~Cuntz and W. ~Krieger},
{\it A class of $C^*$-algebras and topological Markov chains},
 Invent.\ Math.\
 {\bf 56}(1980), pp.\ 251--268.


\bibitem{EKaKa}
{\sc G. K.  Eleftherakis, \, E. T. A.  Kakariadis and E. G. Katsoulis},
{\it  Morita equivalence of C*-correspondences passes to the related operator algebras},
preprint, arXiv:1611.04169 [math OA]. 
%








\bibitem{KaKaTransAMS}
{\sc E. T. A. Kakariadis and E. G. Katsoulis},
{\it Contributions to the theory of $C^*$-correspondences with applications to
multivariable dynamics},
Trans. Amer. Math. Soc. {\bf 364}(2012), pp. \  6605--6630.

\bibitem{KaKaJFA2014}
{\sc E. T. A. Kakariadis and E. G. Katsoulis},
{\it $C^*$-algebras and equivalences for $C^*$-correspondences},
J. Funct. Anal. {\bf 266}(2014), pp. \  956--988.




\bibitem{Kitchens}{\sc B.~P. ~Kitchens},
{\it Symbolic dynamics}, 
Springer-Verlag, Berlin, Heidelberg and New York (1998).





\bibitem{LM}{\sc D. ~Lind and B. ~Marcus},
{\it An introduction to symbolic dynamics and coding},
 Cambridge University Press, Cambridge
(1995).










\bibitem{MaETDS2004} {\sc K. Matsumoto},
{\it Strong shift equivalence of symbolic dynamical systems 
and Morita equivalence of $C^*$-algebras},
Ergodic Theory Dynam. Systems {\bf 24}(2004), pp.\ 199--215.






















\bibitem{MaJOT2015}{\sc K. Matsumoto},
{\it Strongly continuous orbit equivalence of 
one-sided topological Markov shifts},
J. Operator Theory {\bf 74}(2015), pp. 101--127.




\bibitem{MaPre2015}
{\sc K. Matsumoto},
{\it Continuous orbit equivalence, flow equivalence of Markov shifts and circle actions on Cuntz--Krieger algebras},
Math. Z. (2016), DOI:10.1007/s00209-016-1700-3, 21 pages.  

\bibitem{MaPAMS2016}
{\sc K. Matsumoto},
{\it Uniformly continuous orbit equivalence of Markov shifts and gauge actions on Cuntz--Krieger algebras},
Proc. Amer. Math. Soc., AMS Early View, DOI:10.1090/proc/13387,11 pages.  




\bibitem{MaPre2016} {\sc K. Matsumoto},
{\it Topological conjugacy of topological Markov shifts and Cuntz--Krieger algebras},
preprint, arXiv: 1604.02763v4.


\bibitem{MaPre20162} {\sc K. Matsumoto},
{\it Strong shift equivalence of symbolic dynamical systems 
and Morita equivalence of $C^*$-algebras II},
in preparation.










%
\bibitem{MPT}
{\sc P. S. Muhly, D.  Pask and M. Tomforde},
{\it Strong shift equivalence of $C^*$-correspondences},
Israel J. Math. {\bf 167} (2008), pp. \ 315--346. 

%
\bibitem{MS}
{\sc P. S. Muhly and B. Solel},
{\it On the Morita equivalence of tensor algebras},
Proc. London Math. Soc. {\bf 3} (2000), pp. \ 113--168. 



















\bibitem{Rieffel1}
{\sc M. A. Rieffel},
{\it Induced representations of $C^*$-algebras},
Adv.\ in Math.\
{\bf 13}(1974), pp. 176--257.













\bibitem{Ro}
{\sc M. R{\o}rdam},
{\it Classification of Cuntz-Krieger algebras},
 K-theory {\bf 9}(1995), pp.\  31--58.






\bibitem{Tomforde}{\sc M. Tomforde},
{\it Strong shift equivalence in the $C^*$-algebraic setting:
 graphs and $C^*$-correspondences},  
Operator theory, Operator Algebras, and Applications,  221--230, Contemp. Math., {\bf 414}, Amer. Math. Soc., Providence, RI, 2006.












\bibitem{Williams} {\sc R. F. Williams},
{\it Classification of subshifts of finite type}, 
 Ann.\ Math.\  {\bf 98}(1973), pp.\ 120--153.
 erratum, Ann.\ Math.\
{\bf 99}(1974), pp.\ 380--381.

\end{thebibliography}
\end{document}